\documentclass[10pt, english]{amsart}

\usepackage{amsmath,amssymb,enumerate}

\usepackage[T1]{fontenc}
\usepackage[all]{xy}

\usepackage{babel}
\usepackage{amstext}
\usepackage{amsmath}
\usepackage{amsfonts}
\usepackage{latexsym}
\usepackage{ifthen}

\usepackage{xypic}
\xyoption{all}
\usepackage{color}

\newcommand{\cal}{\mathcal}

\renewcommand{\O}{{\mathcal O}}

\newtheorem{lemmaa}{}[section]

\newenvironment{lemma}{\begin{lemmaa}{\bf Lemma.}}{\end{lemmaa}}
\newenvironment{lemmaref}[1]
{\begin{lemmaa}{\bf Lemma\ }{\normalfont #1}{\bf.}}{\end{lemmaa}}
\newenvironment{example}{\begin{lemmaa}{\bf Example.}\rm}{\end{lemmaa}}

\newenvironment{theorem}{\begin{lemmaa}{\bf Theorem.}}{\end{lemmaa}}
\newenvironment{theoremref}[1]
{\begin{lemmaa}{\bf Theorem\ }{\normalfont #1}{\bf.}}{\end{lemmaa}}
\newenvironment{proposition}{\begin{lemmaa}{\bf Proposition.}}{\end{lemmaa}}
\newenvironment{propositionref}[1]
{\begin{lemmaa}{\bf Proposition\ }{\normalfont #1}{\bf.}}{\end{lemmaa}}
\newenvironment{corollary}{\begin{lemmaa}{\bf Corollary.}}{\end{lemmaa}}
\newenvironment{remark}{\begin{lemmaa}{\bf Remark.}\rm}{\end{lemmaa}}

\newenvironment{definition}{\begin{lemmaa}{\bf Definition.}}{\end{lemmaa}}

\newenvironment{remark*}{\noindent{\it Remark.}}{}
\newenvironment{example*}{{\bf Example.}}{}

\renewenvironment{paragraph}[1]
{\setlength{\parskip}{\memoire}
\refstepcounter{lemmaa}
{\medskip\noindent\bf\thelemmaa. #1}}
{\setlength{\parskip}{\medskipamount}
}

\makeatletter
\@addtoreset{equation}{lemmaa}
\renewcommand{\theequation}{\thelemmaa.\arabic{equation}}
\makeatother

\newcommand{\Q}{\ensuremath{\mathbb{Q}}}

\newcommand{\C}{\ensuremath{\mathbb{C}}}
\newcommand{\N}{\ensuremath{\mathbb{N}}}
\newcommand{\PP}{\ensuremath{\mathbb{P}}}

\newcommand{\NE}[1]{ \ensuremath{ \overline { \mbox{NE} }(#1)} }
\newcommand{\NS}[1]{ \ensuremath{ \mbox{NS}(#1) } }

\renewcommand{\epsilon}{\varepsilon}

\newcommand{\holom}[3]{\ensuremath{#1:#2  \rightarrow #3}}
\newcommand{\fibre}[2]{\ensuremath{#1^{-1} (#2)}}

\makeatletter
\ifnum\@ptsize=0  \fi \ifnum\@ptsize=2  \fi 

\newcommand\sO{{\mathcal O}}

\newcommand\Hom{{\rm Hom}}

\DeclareMathOperator*{\Sym}{Sym}

\DeclareMathOperator*{\sing}{sing}

\DeclareMathOperator{\Chow}{Chow}

\DeclareMathOperator*{\Rat}{RatCurves^n}
\newcommand{\Univ}{\ensuremath{\mathcal{U}}}
\DeclareMathOperator*{\loc}{loc}

\newcommand{\minimal}{\mathcal{K}}
\newcommand{\minimalz}{\mathcal{K} ^\circ}
\newcommand{\minimalzt}{\tilde {\mathcal{K}} ^\circ}
\newcommand{\vmrt}{\mathcal{V}}
\newcommand{\vmrtx}{\mathcal{V}_x}
\newcommand{\ev}{\mathrm{ev}}
\newcommand{\binimal}{\ensuremath{\mathcal{W}}}
\newcommand{\HomWz}{\ensuremath {\Hom _\binimal ^\circ}}
\newcommand{\HomWt}{\ensuremath {\Hom _\binimal ^\sim}}

\newlength{\memoire}
\title{Numerical characterisation of quadrics} 
\date{July 28, 2015}

\author{Thomas Dedieu}
\author{Andreas H\"oring}

\address{Thomas Dedieu, Institut de Math\'ematiques de Toulouse (CNRS UMR 5219),
Universit\'e Paul Sabatier,
31062 Toulouse Cedex 9, France
}
\email{thomas.dedieu@m4x.org}

\address{Andreas H\"oring, Laboratoire de Math\'ematiques J.A. Dieudonn\'e,
UMR 7351 CNRS, Universit\'e de Nice Sophia-Antipolis, 06108 Nice Cedex 02, France        
}
\email{hoering@unice.fr}

\begin{document}

\renewcommand{\O}{\mathcal {O}}

\begin{abstract} 
Let $X$ be a Fano manifold
such that $-K_X \cdot C \geq \dim X$ for every rational curve $C \subset X$.
We prove that $X$ is a projective space or a quadric.
\end{abstract}

\maketitle

{\renewcommand{\thelemmaa}{\Alph{lemmaa}}
\section*{Introduction} 

%\subsection{The result}

Let $X$ be a Fano manifold, i.e.\ a  projective manifold with ample
anticanonical bundle. 
If the Picard number of $X$ is at least two, Mori theory shows the
existence of at least two 
non-trivial morphisms $\holom{\varphi_i}{X}{Y_i}$ which contain some
interesting information on the geometry of $X$. 
On the contrary, when the Picard number equals one Mori theory does
not yield any information, and one is 
thus led to studying $X$ in terms of the positivity of the
anticanonical bundle. A well-known example of such  
a characterisation is the following theorem of Kobayashi--Ochiai.

\begin{theoremref}{\cite{KO73}}
\label{t:KO} 
Let $X$ be a projective manifold of dimension $n$. Suppose that $-K_X \simeq d H$ with $H$ an 
ample line bundle on $X$. 
\begin{enumerate}
\item If $d=n+1$, then $X \simeq \PP^n$. 
\item If $d=n$, then $X \simeq \Q^n$.
\end{enumerate}
\end{theoremref}

(Throughout the paper, $\Q^n$ designates a smooth quadric hypersurface
in $\PP^{n+1}$ for any positive integer $n$.)

The divisibility of $-K_X$ in the Picard group is a rather restrictive condition, so it is natural to ask for similar 
characterisations under (a priori) weaker assumptions. Based on
Kebekus' study of singular rational curves \cite{Keb02a}, 
Cho, Miyaoka and Shepherd-Barron proved a generalisation of the first
part of Theorem~\ref{t:KO}:

\begin{theoremref}{\cite{CMS02, Keb02b}}
Let $X$ be a Fano manifold of dimension $n$. Suppose that 
$$
-K_X \cdot C \geq n+1 \quad 
\text{for all rational curves}\ C \subset X.
$$
Then $X \simeq \PP^n$.
\end{theoremref}

The aim of this paper is to prove a similar generalisation for the
second part of Theorem~\ref{t:KO}:

\begin{theorem} \label{t:miyaoka}
Let $X$ be a Fano manifold of dimension $n$. Suppose that
$$
-K_X \cdot C \geq n \quad 
\text{for all rational curves}\ C \subset X.
$$
Then $X \simeq \PP^n$ or $X \simeq \Q^n$.
\end{theorem}

This statement already appears in a paper of Miyaoka
\cite[Thm.0.1]{Miy04}, but the proof there has a gap 
(cf.\ Remark~\ref{remarkgap}).
In this paper we borrow some ideas and tools from Miyaoka's,
yet give a proof based on a completely different strategy.
Note also that Hwang gave a proof under the additional assumption that
the general VMRT (see below) is smooth \cite[Thm.1.11]{Hwa13},
a property that does not hold for every Fano manifold
\cite[Thm.1.10]{CD15}.

In the proof of Theorem~\ref{t:miyaoka}, we have to assume $n \geq 4$;
for $n \leq 3$ the statement follows easily from classification
results.\\
\indent
The assumption that $X$ is Fano assures that
$\rho(X)=1$ because of the Ionescu--Wi\'sniewski inequality
\cite[Thm.0.4]{Ion86}, \cite[Thm.1.1]{Wis91}
(see \S\ref{s:setup}).
It is possible to remove this assumption: the Ionescu--Wi\'sniewski
inequality together with \cite[Thm.1.3]{a18} enable one to deal with
the case $\rho(X)>1$, and one gets the following.

\begin{corollary}
\label{corollary}
Let $X$ be a projective manifold of dimension $n$ containing a
rational curve. If
$$
-K_X \cdot C \geq n \quad 
\text{for all rational curves}\ C \subset X,
$$
then $X$ is a projective space, a hyperquadric,
or a projective bundle over a curve. 
\end{corollary}
\vskip -\medskipamount
\noindent
(Note that under the assumptions of Corollary~\ref{corollary},
if $\rho(X)=1$ then $X$ is Fano.)

\subsection*{Outline of the proof}

In the situation of Theorem~\ref{t:miyaoka} let $\mathcal K$ be a family of minimal rational curves on $X$.
By Mori's bend-and-break lemma a minimal curve $[l] \in \mathcal K$ satisfies $-K_X \cdot l \leq n+1$
and if equality holds then $X \simeq \PP^n$ by \cite{CMS02}. By our assumption we are thus left to deal
with the case $-K_X \cdot l=n$. For a general point $x \in X$ the space $\mathcal K_x$ parametrising
curves in $\mathcal K$ passing through $x$ then has dimension $n-2$ and by \cite[Thm.3.4]{Keb02a}
there exists a morphism
$$
\tau_x : \mathcal K_x \rightarrow \PP(\Omega_{X, x})
$$
which maps a general curve $[l] \in  \mathcal K_x$ to its tangent
direction $T^\perp _{l, x}$.
% \footnote{In this paper we always take the
%   projectivisation in the sense of Grothendieck.}. 
By \cite[Thm.1]{HM04} 
this map is birational onto its image $\vmrtx$, the 
\emph{variety of minimal rational tangents} (VMRT) at $x$. 
We denote by $\vmrt \subset \PP(\Omega_X)$ the total VMRT, i.e.\ 
the closure of the locus covered by the VMRTs $\vmrtx$
for $x \in X$ general. For the proof of Theorem~\ref{t:miyaoka} we
will compute the cohomology class of the total VMRT 
$\vmrt \subset \PP(\Omega_X)$ in terms of the tautological class $\zeta$
and $\pi^* K_X$, where $\holom{\pi}{\PP(\Omega_X)}{X}$ is
the projection map. This computation is based on
the construction, on the manifold $X$, 
of a family $\binimal ^\circ$ of smooth rational curves such that for
every $[C] \in \binimal ^\circ$ one has 
$$
T_X|_C \simeq \O_{\PP^1}(2)^{\oplus n};
$$
it lifts to a family on $\PP(\Omega_X)$
by associating to a curve $C \subset X$ the image $\tilde C$
of the morphism 
$C \to \PP(\Omega_X)$ defined by the invertible quotient 
$$
\Omega_X|_C \rightarrow \Omega_C.
$$
The main technical statement of this paper is:
\begin{proposition} \label{p:zero}
Let $X \not\simeq \PP^n$ be a Fano manifold of dimension $n \geq 4$,
and suppose that 
$$
-K_X \cdot C \geq n \quad 
\text{for all rational curves}\ C \subset X.
$$
Then, in the above notation, one has $\vmrt \cdot \tilde C=0$
for all $[C] \in \mathcal{W} _0$.
\end{proposition}

Once we have shown this statement a similar intersection computation
involving a general minimal rational curve $l$ yields that the 
VMRT $\vmrt_x \subset \PP(\Omega_{X, x})$ is a hypersurface of degree
at most two. We then conclude with some earlier results 
of Araujo, Hwang, and Mok \cite{Ara06, Hwa07, Mok08}.

\subsection*{Acknowledgements.}
We thank St\'ephane Druel for his numerous comments during
this project. This work was partially supported by the A.N.R. project CLASS\footnote{ANR-10-JCJC-0111}.
}

\section{Notation and conventions}

{\setlength{\parskip}{\memoire}
We work over the field $\C$ of complex numbers. 
Topological notions refer to the Zariski topology.

We use the modern notation for projective spaces, as introduced by
Grothendieck: 
if $\cal E$ is a locally free sheaf on a scheme $X$, we let
$\PP(\mathcal{E})$ be $\mathbf{Proj}\,(\Sym\, \mathcal{E})$.
If $L$ is a line in a vector space $V$,
$L ^\perp$ designates the corresponding point in $\PP(V^\vee)$. 

A variety is an integral scheme of finite type over $\C$, 
a manifold is a smooth variety.
A fibration is a proper surjective 
morphism with connected fibres \holom{\varphi}{X}{Y} 
such that $X$ and $Y$ are normal and $\dim X>\dim Y>0$.

\medskip
We will use the standard terminology and results on rational curves, 
as explained in \cite[Ch.II]{Kol96}, \cite[Ch.2,3,4]{Deb01}, and
\cite{Hwa01}.
Let $X$ be a variety.
We remind the reader that following \cite[II, Def.2.11]{Kol96}, the
notation $\Rat X$ refers to the union of the normalisations of those
locally closed subsets of the Chow variety of $X$ parametrising
irreducible rational curves (the superscript $\vphantom{a}^{\rm n}$ is a
reminder that we normalised, and has nothing to do with the dimension).

For technical reasons, we have to consider families of rational curves
on $X$ as living alternately in $\Rat X$ and in $\Hom(\PP^1,X)$. 
Our general policy is to call $\Hom _{\mathcal {R}} \subset 
\Hom(\PP^1,X)$ the family corresponding to $\mathcal {R} \subset
\Rat X$.
}

\section{Preliminaries on conic bundles}

In this section, we establish some basic facts about conic bundles
over a curve and compute some intersection numbers which will turn out
to be crucial for the proof of Proposition~\ref{p:zero}. 
All these statements appear in one form or another in 
\cite[\S 2]{Miy04}, but we recall them and their proofs for the
clarity of exposition. 

\begin{definition} \label{definitionconicbundle}
A \emph{conic bundle} is an equidimensional projective fibration
$\holom{\varphi}{X}{Y}$ %between normal varieties $X$ and $Y$ 
such that there exists a rank three vector bundle $V \rightarrow Y$
and an embedding $X \hookrightarrow \PP(V)$ that maps every
$\varphi$-fibre $\fibre{\varphi}{y}$ onto a conic 
(i.e.\ the zero scheme of a degree $2$ form)
in $\PP(V_y)$.
The set
$$
\Delta := 
\{ 
y \in Y \ | \ \fibre{\varphi}{y} \ \mbox{is not smooth}
\}
$$
is called the \emph{discriminant locus} of the conic bundle.  
\end{definition}

\begin{lemma} \label{lemmadualgraph}
Let $S$ be a smooth surface admitting a projective fibration 
$\holom{\varphi}{S}{T}$ onto a smooth curve such that the general fibre is $\PP^1$, and
such that $-K_S$ is $\varphi$-nef. Let $F$ be a reducible $\varphi$-fibre and suppose that 
$$
F = C_1+C_2+F',
$$
where the $C_i$ are $(-1)$-curves and $C_i \not\subset F'$. Then $F' = \sum E_j$ is a reduced chain of $(-2)$-curves
and the dual graph of $F$ is as depicted in Figure~\ref{fig:dualgraph}.
\begin{figure}[!h]
\caption{}
\label{fig:dualgraph}
\setlength{\unitlength}{1.2mm}
\centering
\begin{picture}(120,10)
\put(25,1){\circle*{2}}
\multiput(35,1)(10,0){4}{\circle{2}}
%\multiput(35,1)(10,0){3}{\circle{2}}
\put(75,1){\circle*{2}}
%\put(45,11){\circle*{2}}\put(45,2){\line(0,1){8}}
\multiput(26,1)(10,0){2}{\line(1,0){8}}
\multiput(46,1)(3,0){3}{\line(3,0){2}}
\multiput(56,1)(10,0){2}{\line(1,0){8}}
\put(24,3){$C_1$}
\put(34,3){$E_1$}
\put(44,3){$E_2$}
\put(54,3){$E_{k-1}$}
\put(64,3){$E_k$}
\put(74,3){$C_2$}
\end{picture}
\end{figure}
\end{lemma}

\begin{proof} 
Write $F' = \sum_{j=1}^k a_j E_j$, $a_j \in \N$, where 
$E_1, \ldots, E_k$ are the irreducible components of $F'$.
Note first that since $-K_S \cdot F=2$ and $-K_S \cdot C_i=1$,
the fact that $-K_S$ is $\varphi$-nef implies 
$-K_S \cdot E_j =0$ for all $j$.
Since $E_j$ is an irreducible component of a reducible fibre,
we have $E_j^2<0$. Thus we see that each $E_j$ is a $(-2)$-curve.

We will now proceed by induction on the number of irreducible
components of $F'$, the case $F'=0$ being trivial.
Let $\holom{\mu}{S}{S'}$ be the blow-down of the
$(-1)$-curve $C_2$;
then by the rigidity lemma \cite[Lemma~1.15]{Deb01}, there is a
morphism $\holom{\varphi'}{S'}{T}$ such that $\varphi = \varphi' \circ
\mu$. 
Note that $S'$ is smooth and $-K_{S'}$ is $\varphi'$-nef
(see, e.g., \cite[Rem.5.7]{a22}).
We also have
$$
0 = C_2 \cdot F = -1 + C_2 \cdot (C_1+\sum_{i=1}^k a_i E_i), 
$$
so $C_2$ meets $C_1+\sum_{i=1}^k a_i E_i$ transversally in exactly one point. If $C_2 \cdot C_1>0$, then
$\mu_*(C_1)$ has self-intersection $0$, yet it is also an irreducible component of the reducible fibre 
$\mu_*(C_1+\sum_{i=1}^k a_i E_i)$, a contradiction. Thus (up to renumbering) we can suppose that
$C_2 \cdot E_1=1$ and $a_1=1$. In particular $\mu_*(E_1)$ is a $(-1)$-curve, so
$$
\mu_*(C_1+\sum_{i=1}^k a_i E_i) = \mu_*(C_1) + \mu_*(E_1) +  \mu_*(\sum_{i=2}^{k} a_i E_i) 
$$
satisfies the induction hypothesis.
\end{proof}

In the following we will use that for every normal surface 
one can define an intersection theory using the Mumford pull-back
to the minimal resolution, cf.\ \cite{Sak84}.

\begin{lemma} \label{lemmadescribeS}
Let $S$ be a normal surface admitting a projective fibration $\holom{\varphi}{S}{T}$ onto a smooth curve
such that the general fibre is $\PP^1$ and such that every fibre is reduced and has at most two irreducible
components. Then 
\begin{enumerate}
\item\label{l:descr-c:cbdle} $\varphi$ is a conic bundle;
\item \label{l:descr-c:Ak} $S$ has at most $A_k$-singularities; and
\item \label{l:descr-c:sing} if $s \in S_{\sing}$, then $s=F_{\varphi(s),1} \cap F_{\varphi(s),2}$ where $F_{\varphi(s)} = F_{\varphi(s),1} + F_{\varphi(s),2}$
is the decomposition of the fibre over $\varphi(s)$ in its irreducible components. In particular $F_{\varphi(s)}$ is a reducible conic.
\end{enumerate}
\end{lemma}

\begin{proof}
If a fibre $\fibre{\varphi}{t}$ is irreducible, then $\varphi$ is a $\PP^1$-bundle in a neighbourhood
of  $\fibre{\varphi}{t}$ \cite[II, Thm.2.8]{Kol96}. Thus we only have to consider points
$t \in T$ such that $S_t:=\fibre{\varphi}{t}$ is reducible. Since
$p_a(S_t)=0$ and $S_t=C_1+C_2$ is reduced, we see that $S_t$ is a
union of two $\PP^1$'s meeting transversally in a point. Since
$S_t=\varphi^* t$ is a Cartier divisor, this already implies 
\ref{l:descr-c:sing}. 
 
Let $\holom{\epsilon}{\hat S}{S}$ be the minimal resolution of the
singular points lying on $S_t$. Then we have
$$
K_{\hat S} \equiv \epsilon^* K_S - E,
$$
with $E$ an effective $\epsilon$-exceptional $\Q$-divisor. Denote by $\hat C_i$ the proper transform of $C_i$.
If $K_{\hat S} \cdot \hat C_i < -1$, then $\hat C_i$ deforms in $\hat S$ \cite[II, Thm.1.15]{Kol96}.
Yet $\hat C_i$ is an irreducible component of a reducible $\varphi \circ \epsilon$-fibre, so this is impossible.
So we have  
$$
K_S \cdot C_i \geq K_{\hat S} \cdot \hat C_i \geq -1
$$
for $i=1,2$. Since $K_S \cdot (C_1+C_2)=-2$, this implies that
$K_S \cdot C_i = -1$ and $E=0$. Thus $S$ has canonical singularities.
Since canonical surface singularities are Gorenstein we see that $-K_S$ is Cartier, $\varphi$-ample and defines an embedding
$$
S \subset \PP(V:=\varphi_*(\sO_S(-K_S)))
$$
into a $\PP^2$-bundle mapping each fibre onto a conic. This proves 
\ref{l:descr-c:cbdle}.

Since $\epsilon$ is crepant, the divisor $-K_{\hat S}$ is $\varphi \circ \epsilon$-nef. Moreover the proper transforms $\hat C_i$ are $(-1)$-curves in $\hat S$.
By Lemma \ref{lemmadualgraph} this proves \ref{l:descr-c:Ak}.
\end{proof}

The following fundamental lemma should be seen as an analogue of the basic fact
that a projective bundle over a curve contains at most one curve with negative self-intersection.

\begin{lemmaref}{\cite[Prop.2.4]{Miy04}}
\label{l:computation}%
\renewcommand{\theequation}{C\arabic{equation}}%
Let $S$ be a normal projective surface that is a conic bundle
$\holom{\varphi}{S}{T}$ over a smooth curve $T$,
and denote by $\Delta$ the discriminant locus.
Suppose that $\varphi$ has two disjoint sections $\sigma_1$ and
$\sigma_2$, both contained in the smooth locus of $S$.
% $$
% \sigma_1 \cup \sigma_2 \subset S_{\nons}.
% $$
Suppose moreover that for every $t \in \Delta$,
the fibre $F_t$ has a decomposition $F_t=F_{t,1} + F_{t, 2}$ such that 
\begin{equation} \label{symmetrysigma}
\sigma_i \cdot F_{t,j} = \delta_{i,j}.
\end{equation}
Eventually, assume that there exists a nef and big divisor $H$ on $S$
such that 
\begin{equation} \label{Hzero}
H \cdot \sigma_1 = H \cdot \sigma_2 = 0.
\end{equation}
Let $\holom{\epsilon}{\hat S}{S}$ be the minimal resolution. 
Let $\sigma$ be a $\varphi$-section, and $\hat \sigma \subset \hat S$ its proper transform. 
Then the following holds:
\begin{enumerate}
\item
\label{c:disjoint-neg}
If $(\hat \sigma)^2<0$, then $\sigma=\sigma_1$ or $\sigma = \sigma_2$.
\item
\label{c:disjoint-zero}
If $(\hat \sigma)^2=0$ then $\sigma$ is disjoint from $\sigma_1 \cup \sigma_2$. 
\end{enumerate}
\setcounter{equation}{0}
\end{lemmaref}

\begin{remark*}
In the situation above the conic bundle does not have any non-reduced
fibre, since there exists 
a section that is contained in the smooth locus.
\end{remark*}

\begin{proof}
{\em Preparation: contraction to a smooth ruled surface.}
Lemma~\ref{lemmadescribeS} applies to the surface $S$.
It follows that $S$ has an $A _{k_t}$-singularity ($k_t \geq 0$) in 
$F _{t,1} \cap F _{t,2}$ for every $t \in \Delta$, and no further 
singularity.
% the latter has $A_k$-double points as its only
% possible singularities, and that each of those is a point where
% $\varphi$ is not smooth.
In particular, the dual graph of $\fibre{(\varphi \circ \epsilon)}{t}$
is as described in Lemma~\ref{lemmadualgraph}
for every $t \in \Delta$.
%we let $k_t$ be the non-negative integer such that $S$ has

We consider the birational morphism 
\[
\hat \mu : \hat S \to S ^\flat
\]
defined as the composition, for every $t \in \Delta$, of the blow-down
of the proper transform ${\hat F}_{t,1}$ of $F _{t,1}$  and of all the
$k_t$ $(-2)$-curves contained in 
$(\varphi \circ \epsilon) ^{-1} (t)$.
Since $\hat \mu$ is a composition of blow-down of $(-1)$-curves, the
surface $S ^\flat$ is smooth.
By the rigidity lemma \cite[Lemma 1.15]{Deb01}, there is a morphism 
$\varphi ^\flat : S^\flat \to T$. All its fibres are irreducible
rational curves, so it is a $\PP ^1$-bundle by \cite[II,
Thm.2.8]{Kol96}.
Again by the rigidity lemma, $\hat \mu$ factors through $\epsilon$, 
i.e. there is a birational morphism $\mu : S \to S ^\flat$ such that 
$\hat \mu = \mu \circ \epsilon$; it is the contraction of all the curves
$F _{t,1}$, $t \in \Delta$.

Since $\sigma_1$ meets $F_{t, 1}$ in a smooth point of $S$, the proper
transforms $\hat \sigma_1$ and $\hat F_{t,1}$ meet in the same point. 
Thus (the successive images of) $\hat \sigma_1$ meets the exceptional
divisor of all the blow-downs of $(-1)$-curves composing $\hat \mu$,
and since the section 
$\sigma_1 ^\flat := \hat \mu (\hat \sigma_1)$ is
smooth, all the intersections are transversal. 
Vice versa we can say that $\hat S$ is obtained from $S ^\flat$ by
blowing up points on (the proper transforms of) $\sigma_1 ^\flat$. 

By the symmetry condition
\eqref{symmetrysigma} the curve $\sigma_2$ is disjoint from the
$\mu$-exceptional locus, so if we set $\sigma_2
^\flat:=\mu(\sigma_2)$,
then we have
$(\sigma_2 ^\flat)^2 = (\sigma_2)^2$.
Since $H \cdot \sigma_2=0$ the Hodge index theorem implies
$( \sigma_2 ^\flat)^2 =\sigma_2^2 <0$. 
In the notation of \cite[V,Ch.2]{Har77} 
$\holom{\varphi ^\flat}{S ^\flat}{T}$
is a ruled surface with invariant $-e :=  (\sigma_2 ^\flat)^2 > 0$.
In particular the Mori cone $\NE{S ^\flat}$ is generated by 
a general $\varphi ^\flat$-fibre $F$ and $\sigma_2 ^\flat$.
Since $\sigma_1 ^\flat \cdot \sigma_2 ^\flat = 0$
and $\sigma_1 ^\flat \cdot F =1$, we have
\begin{equation} \label{compare}
\sigma_1 ^\flat \equiv \sigma_2 ^\flat + e F.
\end{equation}

\noindent
{\em Conclusion.}
Let now $\sigma \subset S$ be a section that is distinct from both
$\sigma_1$ and $\sigma_2$. Then $\sigma ^\flat:=\mu(\sigma)$ 
is distinct from both $\sigma_1 ^\flat$ and $\sigma_2 ^\flat$. 
Since $\sigma ^\flat \neq \sigma_2 ^\flat$ we have
\begin{equation} \label{help1}
\sigma ^\flat \equiv \sigma_2 ^\flat + c F
\end{equation}
for some $c \geq e$ \cite[V, Prop.2.20]{Har77}.
Since $\sigma ^\flat \neq \sigma_1 ^\flat$ we have
\begin{equation} \label{help2}
\sigma ^\flat \cdot \sigma_1 ^\flat \geq \sum_{t \in \Delta} \tau_t,
\end{equation}
where $\tau_t$ is the intersection multiplicity of
$\sigma ^\flat$ and $\sigma_1 ^\flat$ at the point 
$F_t \cap \sigma_1 ^\flat$.
Denote by $\hat \sigma \subset \hat S$ the proper transform of 
$\sigma \subset S$, which is also the proper transform
of $\sigma ^\flat \subset S ^\flat$.
By our description of $\hat \mu$ as a sequence of blow-ups in 
$\sigma_1 ^\flat$ we obtain
\[
(\hat \sigma)^2 = (\sigma ^\flat)^2 - 
\sum_{t \in \Delta} \min (\tau_t, k_t+1)
\geq (\sigma ^\flat)^2 - \sum_{t \in \Delta} \tau_t.
\]
By \eqref{help2} this implies
\[
(\hat \sigma)^2 
\geq (\sigma ^\flat)^2 - \sigma ^\flat \cdot \sigma_1 ^\flat 
=  \sigma ^\flat \cdot (\sigma ^\flat - \sigma_1 ^\flat). 
\]
Plugging in \eqref{compare} and \eqref{help1} we obtain
\begin{equation}
\label{help3}
(\hat \sigma)^2 \geq c-e \geq 0.
\end{equation}
This shows statement a). 

Suppose now that $(\hat \sigma)^2=0$. Then by \eqref{help3} we have $c=e$,
hence $\sigma ^\flat \cdot \sigma_2 ^\flat=0$.
Being distinct, the two curves $\sigma ^\flat$ and $\sigma_2 ^\flat$
are therefore disjoint, and so are
their proper transforms $\hat \sigma$ and $\hat \sigma_2$.
Note now that $\epsilon$ is an isomorphism in a neighbourhood of $\hat \sigma_2$, so $\sigma = \epsilon(\hat \sigma)$
is disjoint from $\sigma_2=\epsilon(\hat \sigma_2)$.
In order to see that $\sigma$ and $\sigma_1$ are disjoint, we repeat the same argument but contract those fibre
components which meet $\sigma_2$. This proves statement b). 
\end{proof}

\section{The main construction} \label{S:work}

\paragraph{Set-up.}
\label{s:setup}
For the whole section, we
let $X \not \simeq \PP^n$ be a Fano manifold of dimension $n \geq 4$,
and suppose that
\begin{equation} \label{greatern}
-K_X \cdot C \geq n \quad 
\text{for all rational curves}\ C \subset X.
\end{equation}
This is the situation of Proposition \ref{p:zero};
let us show that it implies that the Picard number $\rho(X)$ equals 
$1$.

The Mori cone $\NE X$ generates $\NS X$ as a vector
space, and since $X$ is Fano the number of extremal rays of $\NE X$ is therefore at
least $\rho(X)$. Now the Ionescu--Wi\'sniewski inequality
\cite[Thm.0.4]{Ion86}, \cite[Thm.1.1]{Wis91} implies that the
contraction of any extremal ray of $\NE X$ is of fibre type, with a
base of dimension at most $1$. If $\rho(X)>1$, there are therefore at
least two Mori fibrations, with bases of dimension $1$; but then their
respective fibres are divisors in $X$, and since $n>2$ their
intersection contains a curve that is contracted by both fibrations, a
contradiction. 
\endparagraph

Recall that a family of 
\emph{minimal rational curves}
is an irreducible component $\minimal$ of $\Rat(X)$
such that the curves in $\minimal$ dominate $X$,
and for $x \in X$ general 
the algebraic set $\minimal_x ^\flat \subset \minimal$ parametrising curves passing through $x$ is proper.
We will use the following simple observation:

\begin{lemma} \label{l:nminimal}
In the situation of Proposition~\ref{p:zero}, 
let $l \subset X$ be a rational curve such that $-K_X \cdot l=n$. 
Then any irreducible component $\minimal$ of $\Rat X$ containing $[l]$
is a family of minimal rational curves.
\end{lemma}

\begin{proof}
Condition~\eqref{greatern} implies the properness of $\minimal$
\cite[II, (2.14)]{Kol96}, so
we only have to show that the deformations of $l$ dominate $X$.
We have $\dim \minimal \geq 2n-3$ by \cite[II, Thm.1.2,
Thm.2.15]{Kol96}. 
Thus if $\minimal$ is not covering, then for a point $x \in X$ such
that $\minimal_x ^\flat \neq \emptyset$  
one has $\dim \minimal_x ^\flat \geq n-1$. Yet by the bend-and-break lemma
\cite[Prop.3.2]{Deb01}
the universal family over $\minimal_x ^\flat$ 
has a generically finite map to $X$. Thus the curves in 
$\minimal_x ^\flat$ dominate $X$, a contradiction.  
\end{proof}

\paragraph{Minimal rational curves and VMRTs.}
\label{s:minimal}
Since $X$ is Fano, it contains a rational curve $l$.
Since $X \not \simeq \PP^n$, there exists a rational
curve with $-K_X \cdot l=n$ by \cite{CMS02},
and by Lemma~\ref{l:nminimal} there exists a family of
minimal rational curves containing the point $[l] \in \Rat(X)$.
We fix once and for all such a family, which we call $\minimal$.

For $x \in X$ general, denote by $\minimal _x$ the normalisation of
the algebraic set $\minimal_x ^\flat \subset \minimal$ parametrising
curves passing through $x$.
Every member of $\minimal_x ^\flat$ is a free curve
(this follows from the argument of
\cite[II, proof of Thm.3.11]{Kol96}),
so $\minimal _x$ is smooth and has dimension $n-2 \geq 2$
\cite[II, (1.7) and (2.16)]{Kol96}.

By results of Kebekus, 
a general curve $[l] \in \minimal _x ^\flat$ 
is smooth \cite[Thm.3.3]{Keb02a}, 
and the \emph{tangent map}
$$
\tau_x : \minimal _x \rightarrow \PP(\Omega_{X, x})
$$
which to a general curve $[l]$ associates its tangent direction 
$T^\perp _{l, x}$ is a finite morphism \cite[Thm.3.4]{Keb02a}.
Its image $\vmrtx$ is called the 
\emph{variety of minimal rational tangents} (VMRT) at $x$.
The map $\tau_x$ is birational by \cite[Thm.1]{HM04},
so the normalisation of $\vmrtx$ is $\minimal _x$, which is smooth
(this is \cite[Cor.1]{HM04}).
Also, one can associate to a general point $v \in \vmrtx$
a unique minimal curve $[l] \in \minimal _x$. 
We denote by $\vmrt \subset \PP(\Omega_X)$ the 
\emph{total VMRT},
i.e.\ the closure of the locus covered by the VMRTs $\vmrtx$
for $x \in X$ general. Since $\minimal _x$ has dimension $n-2$, the
total VMRT $\vmrt$ is a divisor in $\PP(\Omega_X)$. 

For a general $[l] \in \minimal$, one has
\begin{equation} \label{splitstandard}
T_X|_l \simeq \sO_{\PP^1}(2) \oplus \sO_{\PP^1}(1)^{\oplus n-2} 
\oplus \sO_{\PP^1}
\end{equation}
\cite[Exercise~4.8.3]{Deb01}.
We call a minimal rational curve $[l] \in \minimal$ 
\emph{standard} if $l$ is smooth and the bundle 
$T_X|_l$ has the same splitting type as in \eqref{splitstandard}.
\endparagraph

\begin{paragraph}{Smoothing pairs of minimal curves.}
For a general point $x_1 \in X$, by the bend-and-break lemma
\cite[Prop.3.2]{Deb01}
the curves parametrised by $\minimal_{x_1}$ cover
a divisor $D_{x_1} \subset X$. 
This divisor is ample because $\rho(X)=1$, so for $x_2 \in X$ 
and $[l_2] \in \minimal_{x_2}$ 
the curve $l_2$ intersects $D_{x_1}$.
Thus for a pair of general points $x_1, x_2 \in X$ we can find 
a chain of two standard curves $l_1 \cup l_2$ connecting
the points $x_1$ and $x_2$. 
\par Let $\mbox{loc}^1_{x_1}$ be
the locus covered by {\em all}
the minimal rational curves of $X$ passing through $x_1$.
It is itself a divisor, but may be bigger than $D_{x_1}$
since in general there are finitely many families of minimal curves.
We choose a pair of general points $x_1,x_2 \in X$
such that $x_2 \not\in \mbox{loc}^1_{x_1}$ 
(which implies $x_1 \not\in \mbox{loc}^1_{x_2}$),
and consider a chain of two standard curves $l_1 \cup l_2$ connecting
the points $x_1$ and $x_2$ as above.
\par By \cite[II, Ex.7.6.4.1]{Kol96} the union
$l_1 \cup l_2$ is dominated by a transverse union $\PP^1 \cup \PP^1$.
Since both rational curves are free
we can smooth the tree $\PP^1 \cup \PP^1$ keeping the point $x_1$ fixed \cite[II, Thm.7.6.1]{Kol96}. 
Since $x_1$ is general in $X$ this defines a family of rational curves
dominating $X$, 
and we denote by $\binimal$ the normalisation of the irreducible
component of $\Chow(X)$ containing these rational curves. 
\end{paragraph}

\paragraph{}%
Since a general member $[C]$ of the family $\binimal$ is free
and $-K_X\cdot C =2n$,
we have $\dim \binimal = 3n-3$.
We pick an arbitrary irreducible component of the
subset of $\binimal$ parametrising cycles containing $x_1$, and let
$\binimal _{x_1}$ be its normalisation; then we have
$\dim \binimal _{x_1} = 2n-2$.
 Let $\Univ _{x_1}$ be the normalisation of the universal
family of cycles over $\binimal _{x_1}$.
The evaluation map $\ev _{x_1}: \Univ _{x_1} \to X$ is
surjective:
its image is irreducible, and it contains both 
the divisor $D _{x_1}$
(because it is contained in the image of the restriction of 
$\ev_{x_1}$ to those members of $\binimal_{x_1}$ that contain a
minimal curve through $x_1$)
and the point $x_2 \not\in D _{x_1}$.
\par Next, 
we choose an arbitrary irreducible component
of the subset of $\binimal$ parametrising cycles passing through $x_1$
and $x_2$, 
and let $\binimal_{x_1, x_2}$ be its normalisation,
$\Univ_{x_1, x_2}$ the normalisation of the universal family over
$\binimal_{x_1, x_2}$.
We denote by
$$
\holom{q}{\Univ_{x_1, x_2}}{\binimal_{x_1, x_2}},
\qquad
\holom{\ev}{\Univ_{x_1, x_2}}{X}
$$
the natural maps. 
It follows from the considerations above that $\binimal _{x_1,x_2}$ is
non-empty of dimension $n-1$.

By construction, a general curve $[C] \in \binimal_{x_1, x_2}$ is smooth
at $x_i$, $i \in \{1,2\}$,
so the preimage $\fibre{\ev}{x_i}$ contains a unique divisor
$\sigma_i$ that surjects onto $\binimal_{x_1, x_2}$. 
Since $\ev$ is finite on the $q$-fibres and $\binimal_{x_1, x_2}$ is normal, 
we obtain that the degree one map $\sigma_i \rightarrow \binimal_{x_1, x_2}$ is an isomorphism.  
We call the divisors $\sigma_i$ the distinguished sections of $q$.
\endparagraph

Let $\Delta \subset \binimal_{x_1, x_2}$ be the locus
parametrising non-integral cycles. 

\begin{lemma} \label{l:discriminant}
In the situation of Proposition \ref{p:zero} and using the notation
introduced above, let  
$$
C = \sum a_i l_i
$$
be a non-integral cycle corresponding to a point $[C] \in \Delta$. 
Then $C=l_1 + l_2$, with the $l_i$ minimal rational curves
such that $x_i \in l_j$ if and only if $i=j$.
\end{lemma}

\begin{remark*}
Note that we do not claim that the curves $l_i$ belong to the family $\minimal$. However by construction
of the family $\binimal$ as smoothings of pairs $l_1 \cup l_2$ in $\minimal$ there exists an irreducible component $\Delta_{\minimal} \subset \Delta$ such that $l_i \in \minimal$.
\end{remark*}

\begin{proof} 
By \cite[II, Prop.2.2]{Kol96} all the irreducible components $l_i$ are rational curves.
We can suppose that up to renumbering one has $x_1 \in l_1$. 
If $a_1 \geq 2$, then $-K_X \cdot C=2n$ and $-K_X \cdot l_1 \geq n$ implies
that $C= 2 l_1$ and $l_1$ is a minimal rational curve. Yet this contradicts the assumption $x_2 \not\in \mbox{loc}^1_{x_1}$.
Thus we have $a_1=1$ and since $C$ is not integral there exists a second irreducible component $l_2$. 
Again $-K_X \cdot C=2n$ and $-K_X \cdot l_i \geq n$ implies $C=l_1 + l_2$ and the $l_i$ are minimal rational curves by Lemma \ref{l:nminimal}.
The last property now follows by observing that $x_2 \not\in \mbox{loc}^1_{x_1}$ implies that $x_1 \not\in \mbox{loc}^1_{x_2}$.
\end{proof}

By \cite[II, Thm.2.8]{Kol96},
the fibration \holom{q}{\Univ_{x_1, x_2}}{\binimal_{x_1, x_2}} is a
$\PP^1$-bundle over the open set $\binimal_{x_1, x_2} \setminus
\Delta$. Although Lemma \ref{l:discriminant} essentially says 
that the singular fibres are reducible conics, it is a priori not clear
that $q$ is a conic bundle 
(cf.\ Definition~\ref{definitionconicbundle}). 
This becomes true after we make a base change to a smooth curve.

\begin{lemma} \label{l:basechange}
In the situation of Proposition \ref{p:zero} and using the notation introduced above,
let $Z \subset \binimal_{x_1, x_2}$ be a curve
such that a general point of $Z$ parametrises an irreducible curve. 
Then there exists a finite morphism
$T \rightarrow Z$ such that the normalisation $S$ of the fibre product $\Univ_{x_1, x_2} \times_{\binimal_{x_1, x_2}} T$ has a conic bundle
structure $\varphi: S \rightarrow T$ that satisfies the conditions of Lemma \ref{l:computation}. 
\end{lemma}

\begin{proof}
\setlength{\parskip}{0mm}
Let $\nu: \tilde Z \rightarrow Z$ be the normalisation, and let $N$
be the normalisation of $\Univ_{x_1, x_2}
\times_{\binimal_{x_1, x_2}}  \tilde Z$,
$f_N: N \rightarrow X$ the morphism induced by
$\ev: \Univ_{x_1, x_2} \rightarrow X$. Since all the curves
pass through $x_1$ and $x_2$ there exists a curve $Z_1 \subset N$
(resp. $Z_2 \subset N$) that is contracted by $f_N$ onto the point
$x_1$ (resp. $x_2$). Since $\ev$ is finite on the $q$-fibres, the curves
$Z_1$ and $Z_2$ are multisections 
of $N \rightarrow \tilde Z$. If $\tilde Z_i$ is the normalisation of
$Z_i$, then the fibration $(N \times_{\tilde Z} \tilde Z_i)
\rightarrow \tilde Z_i$ has a section 
given by $c \mapsto (c,c)$. Thus there exists a finite base change $T
\rightarrow \tilde Z$ such that the normalisation 
$\varphi: S \rightarrow T$ of the fibre product $(\Univ_{x_1, x_2} \times_{\binimal_{x_1, x_2}} T) \rightarrow T$ has a natural morphism
$f: S \rightarrow X$ induced by $\ev: \Univ_{x_1, x_2} \rightarrow X$
and contracts two $\varphi$-sections $\sigma_1$ and $\sigma_2$
on $x_1$ and $x_2$ respectively.

Since $Z \not\subset \Delta$, the general $\varphi$-fibre is
$\PP^1$. Moreover by Lemma~\ref{l:discriminant} all the
$\varphi$-fibres are reduced and have at most two irreducible
components. By Lemma~\ref{lemmadescribeS} this implies that $\varphi$
is a conic bundle 
and if $s \in S_{\sing}$, then $F_{\varphi(s)}$ is a reducible conic and the two irreducible components meet in $s$.
Thus we have $\sigma_i \subset S_{nons}$, since otherwise both irreducible components would pass through $x_i$,
thereby contradicting the property that $x_2 \not\in  \loc^1_{x_1}$. 
For the same reason we can decompose any reducible $\varphi$-fibre $F_t$  
by defining $F_{t,i}$ as the unique component meeting the section $\sigma_i$. Since $\sigma_i \cdot F=1$ for a general $\varphi$-fibre
we see that \eqref{symmetrysigma} holds. Condition~\eqref{Hzero}
holds with $H$ the pull-back of an ample divisor on $X$.  
\end{proof}

From this one deduces with Lemma~\ref{l:computation} the following
statement, in the spirit of the 
bend-and-break lemma \cite[Prop.3.2]{Deb01}.

\begin{lemma}
\label{l:bandb}
The restriction of the evaluation map 
$\ev: \Univ_{x_1, x_2} \rightarrow X$ to the complement of 
$\sigma_1 \cup \sigma_2$ is quasi-finite.
In particular $\ev$ is generically finite onto its image.
% and its restriction to the preimage of the complement of
% $\{x_1,x_2\}$ is finite
\end{lemma}

\begin{proof}
We argue by contradiction. Since $\ev$ is finite on the $q$-fibres there exists a curve
$Z \subset \binimal _{x_1, x_2}$ such that the natural map from the
surface $\fibre{q}{Z}$ onto $\ev (\fibre{q}{Z})$ 
contracts three disjoint curves $\sigma_1, \sigma_2$ and $\sigma$ onto
the points $x_1, x_2$ and $x$. \\
\indent
If $Z \not\subset \Delta$, then by Lemma~\ref{l:basechange} 
we can suppose, possibly up to a finite base change,
that $\fibre{q}{Z} \rightarrow Z$
satisfies the conditions \eqref{symmetrysigma} of Lemma~\ref{l:computation}.
After a further base change we can assume that $\sigma$ is a section.
Since $\sigma$ is
contracted by $\ev$ we have $\sigma^2<0$.  
By Lemma~\ref{l:computation},\ref{c:disjoint-neg},
this implies $\sigma=\sigma_1$ or
$\sigma=\sigma_2$, a contradiction. \\
\indent
If $Z \subset \Delta$, then all the fibres over $Z$ are unions of two
minimal rational curves. Thus the normalisation of $\fibre{q}{Z}$ is a
union of two $\PP^1$-bundles mapping onto $Z$ and by construction they
contain three curves which are mapped onto points. 
However a ruled surface contains at most one contractible curve, a
contradiction. 
\end{proof}

\paragraph{}%
Since $\dim \Univ _{x_1,x_2} = \dim X$, one deduces from
Lemma~\ref{l:bandb} above that the cycles $[C] \in \binimal$ 
passing through $x_1, x_2$ cover the manifold $X$. 
By \cite[4.10]{Deb01} 
this implies that a general member $[C] \in \binimal_{x_1, x_2}$
is a $2$-free rational curve \cite[Defn.4.5]{Deb01}.
Since $-K_X \cdot C=2n$, this forces
\begin{equation} \label{splittingconic}
f^* T_X \simeq \sO_{\PP^1}(2)^{\oplus n},
\end{equation}
where $f: \PP^1 \rightarrow C \subset X$ is the normalisation of $C$.
As a consequence, one sees from \cite[II, Thm.3.14.3]{Kol96} that a
general member $[C] \in \binimal$ is a {\em smooth}
rational curve in $X$. 

Let $\HomWz \subset \mbox{Hom}(\PP^1, X)$ be the irreducible open set
parametrising morphisms 
$f: \PP^1 \rightarrow X$ such that the image $C:=f(\PP^1)$ is smooth,
the associated cycle $[C] \in \Chow(X)$ 
is a point in $\binimal$, and $f^* T_X$ has the splitting type
\eqref{splittingconic}. 
By what precedes, the image of $\HomWz$ in $\binimal$ under the
natural map $\mbox{Hom}(\PP^1, X) \rightarrow \Chow(X)$ is a dense
open set $\binimal^\circ \subset \binimal$.
\endparagraph

\paragraph{}%
Denote by $\holom{\pi}{\PP(\Omega_X)}{X}$ the projection map.
We define an injective map 
$$
i: \HomWz \hookrightarrow \mbox{Hom}(\PP^1, \PP(\Omega_X))
$$
by mapping $f: \PP^1 \rightarrow X$ 
to the morphism $\tilde {f}: \PP^1 \to \PP(\Omega_X)$ corresponding to
the invertible quotient $f^* \Omega_X \to \Omega _{\PP^1}$.
Correspondingly, for $[C] \in \binimal ^\circ$ with normalisation $f$, 
we call $[\tilde C]$ the member of $\Chow(\PP(\Omega_X))$ corresponding
to the lifting $\tilde f$.

We let $\HomWt$ be the image of $i$.
Note that it parametrises a family of rational curves that dominates
$\PP(\Omega_X)$, 
but it is not an irreducible component of $\mbox{Hom}(\PP^1,
\PP(\Omega_X))$. 
Indeed, $\HomWt$ is contained in a (much bigger) irreducible component
defined by morphisms corresponding to arbitrary quotients 
$f^* \Omega_X \twoheadrightarrow \sO_{\PP^1}(-2)$.
\endparagraph

The following property is well-known to experts.
Since $\HomWt$
is not an open set of the Hom-space $\mbox{Hom}(\PP^1,
\PP(\Omega_X))$, we have to adapt the proof of
\cite[II,Prop.3.7]{Kol96}.

\begin{lemma} \label{l:badset}
In the situation of Proposition~\ref{p:zero},
let $\vmrt_0 \subset \vmrt$ be a dense, Zariski open set in the total VMRT $\vmrt$, and let $\tilde C:=\tilde f(\PP^1)$ 
be a rational curve parametrised
by a general point of $\HomWt$. Then one has 
$$
(\vmrt \cap \tilde C) \subset (\vmrt_0 \cap \tilde C).
$$
\end{lemma}

\begin{proof}\setlength{\parskip}{0mm}
Set $Z:= \vmrt \setminus \vmrt_0$. 
A point $z \in \PP(\Omega_X)$ is $z=(v_z^\perp,x)$, where
$\C v_z \subset T_{X,x}$ is a tangent direction in $X$ at $x=\pi(z)$.
So for all $p \in \PP^1$, $z=(v_z^\perp,x) \in \PP(\Omega_X)$,
the morphisms $[\tilde f] \in \HomWt$
mapping $p$ to $z$ correspond to morphisms $f: \PP^1
\rightarrow X$ in $\HomWz$ mapping $p$ to $x$ 
with tangent direction $\C v_z$.
Since $f$ has the splitting type \eqref{splittingconic},
the set of these morphisms has dimension exactly $n$.
It follows that
\[
\Hom _{\binimal,Z} ^\sim :=
\bigl\{ 
[\tilde f] \in \HomWt \ | \ \tilde f(\PP^1) \cap Z \neq \emptyset 
\bigr\}
=
\bigcup_{z \in Z} \bigcup_{p \in \PP^1}
\bigl\{
[\tilde f] \in \HomWt \ | \tilde f(p) = z
\bigr\}
\]
has dimension at most $\dim Z + 1 +n$.

Now $\vmrt \subset \PP(\Omega_X)$
is a divisor, and $Z$ has codimension at least one 
in $\vmrt$, so $Z$ has 
% codimension at least two in $\PP(\Omega_X)$. 
dimension at most $2n-3$,
and the set $\Hom _{\binimal,Z} ^\sim$ above has dimension at most
$3n-2$.
Since $\HomWz$ has dimension $3n$ and $\HomWz \rightarrow \HomWt$ is
injective, a general point $[\tilde f] \in \HomWt$ is not in 
$\Hom _{\binimal,Z} ^\sim$.
\end{proof}

We need one more technical statement:

\begin{lemma} \label{l:notcontained}
In the situation of Proposition \ref{p:zero} and using the notation
introduced above, let $[f] \in \HomWz$ be a general point. Then for
\emph{every} $x \in f(\PP^1)$ we have $f(\PP^1) \not\subset \loc^1_{x}$.  
\end{lemma}

\begin{proof} Fix two general points $x_1, x_2 \in X$. A general morphism $[f] \in \HomWz$ passing through $x_1$ and $x_2$ is $2$-free and up to reparametrisation we have $f(0)=x_1, f(\infty)=x_2$. Set $g:=f|_{\{ 0, \infty \}}$, then $f$ is free over $g$ 
\cite[II, Defn.3.1]{Kol96}. Suppose now that such a curve has the property $f(\PP^1) \subset \loc^1_{x_0}$
for some $x_0 \in f(\PP^1)$. Thus $x_1, x_2 \in \loc^1_{x_0}$, hence by symmetry $x_0 \in (\loc^1_{x_1} \cap \loc^1_{x_2})$. 
Yet the intersection 
$$
\loc\nolimits ^1_{x_1} \cap \loc\nolimits ^1_{x_2}
$$
has codimension two in $X$. By \cite[II, Prop.3.7]{Kol96} a general deformation of $f$ over $g$ is disjoint from this set. 
\end{proof}

\paragraph{Proof of Proposition \ref{p:zero}.}
Arguing by contradiction, we suppose that $\vmrt \cdot \tilde C>0$
($\tilde C$ is not contained in $\vmrt$ for the general $[C] \in
\binimal ^\circ$).
Applying Lemma~\ref{l:badset} with 
$$
\vmrt_0 := \{ 
v^\perp \in \vmrt \ | \ \C v = T_{l, \pi(v)} \ 
\text{where $[l] \in \mathcal
  K$ is standard}
\},
$$
we see that for a general point $[C] \in \binimal$ there exists a 
point $x_1 \in C$  and a standard curve $[l] \in \minimal_{x_1}$
such that 
\begin{equation} \label{sametangent}
T_{C, x_1} = T_{l, x_1}.
\end{equation}

We shall now reformulate the property \eqref{sametangent}
in terms of the universal family $\Univ_{x_1,  x_2}$,
with $x_2$ a point chosen in $C \setminus \mbox{loc}^1_{x_1}$
thanks to Lemma~\ref{l:notcontained}.
Consider the blow-up $\epsilon: \tilde X \to X$ at the point $x_1$, with
exceptional divisor $E_1$. Since the general member of 
$\binimal _{x_1,x_2}$ is smooth at $x_1$, there is a rational map 
$\tilde \ev: \Univ _{x_1,x_2} \dashrightarrow \tilde X$
such that $\epsilon \circ \tilde \ev = \ev$ (on the locus where $\tilde \ev$ is defined). It restricts
to a well-defined rational map 
$\sigma_1 \dashrightarrow E_1$ 
which is dominant, and therefore generically finite,
because the general member of $\binimal _{x_1,x_2}$
is $2$-free.
In particular we may assume it is finite in a neighbourhood of the 
point $C \cap \sigma_1$.

We then consider the proper transform $\tilde l$ of $l$ under
$\epsilon$, and let $\Gamma$ be an irreducible component of 
$\tilde \ev ^{-1} (\tilde l)$ passing through $C \cap \sigma_1$.
It is a curve that is mapped to a curve in $\binimal _{x_1, x_2}$ by
$q$.
Also, applying the same construction to the divisor 
$D _{x_1} \subset X$, one gets a prime divisor 
$G \subset \Univ _{x_1,x_2}$ mapping surjectively onto 
$D _{x_1}$ and $\binimal _{x_1,x_2}$ respectively.

Since both maps $q|_G$ and $\ev |_G$ are étale at the general point
of $G$, for the general $l' \in \minimal _{x_1}$
there exists a curve $\Gamma' \subset G$ such that $q(\Gamma')$ is
a curve, $\ev(\Gamma')=l'$,
both maps $q|_G$ and $\ev |_G$ are étale at the general point 
$x \in l'$, and there is $[C'] \in \binimal ^\circ$ such that
the point $C' \cap \sigma_1$ lies on $\Gamma'$.
Yet this is a contradiction to Proposition~\ref{p:separation} below.
\qed
\endparagraph

\begin{propositionref}{\cite[Lemma~3.9]{Miy04}}
\label{p:separation} 
In the situation of Proposition~\ref{p:zero},
let $x_1, x_2 \in X$ be general points, and $[l]$ a general member
of $\minimal _{x_1}$.
Consider an irreducible curve
$\Gamma \subset \Univ_{x_1, x_2}$ such that $\ev(\Gamma)=l$ and 
$q(\Gamma)$ is a curve,
and assume there exists a prime divisor $G \subset \Univ _{x_1,x_2}$
mapped onto $D_{x_1}$ by $\ev$ and containing $\Gamma$, such that both 
maps $q|_G$ and $\ev|_G$ are étale at a general point of $\Gamma$.
Then $\Gamma \cap \sigma_1$ does not contain any point 
$C \cap \sigma_1$ with $[C] \in \binimal ^\circ$.
\end{propositionref}

We give the proof for the sake of completeness.

\begin{proof}
Since $[l]$ is general in $\minimal _{x_1}$, we have 
$$
T_X|_{l} \simeq \sO_{\PP^1}(2) \oplus \sO_{\PP^1}(1)^{n-2} \oplus \sO_{\PP^1},
$$
and $\minimal_{x_1}$ is smooth 
with tangent space $H^0(l, N_{l/X}^+ \otimes \O_{l}(-x_1))$ at $[l]$,  
where $\mathcal {E}^+$ denotes the ample part of a vector bundle  
$\mathcal{E} \rightarrow \PP^1$, i.e.\ 
its ample subbundle of maximal rank.

Let $x \in \Gamma$ be a general point, and set $y= \ev (x) \in l$.
For some analytic neighbourhood $V \subset \minimal_{x_1}$
of $[l]$,
we have an evaluation map
\[
\PP^1 \times V 
\longrightarrow D_{x_1}
\]
which is \'etale at $(y,[l])$,
and the tangent space to $D _{x_1}$ at $y$ is thus
\begin{equation*}
T_{D_{x_1},y} =
T _{l,y} \oplus \bigl( N_{l/X}^+ \otimes \O_{l}(-x_1)
\bigr)_y
= T_X|_{l, y} ^+.
\end{equation*}
Since $q|_G$ and $\ev|_G$ are \'etale in $x$ we obtain that
\begin{equation} \label{idsubspaces}
(q^* T_{\binimal_{x_1, x_2}})_{x} = \ev^* T_X|_{l, \ev(x)}^+
\end{equation}
as subspaces of $T_{\Univ_{x_1, x_2},x}$.

We argue by contradiction and suppose that there
exists $[C] \in \binimal ^\circ$ such that $(C \cap \sigma_1) \in
(\Gamma \cap \sigma_1)$. 
Since $\Gamma$ maps onto $l$ it is not contained in the divisor
$\sigma_1$. 
Since the smooth rational curve $C$ is $2$-free, there exists by
semicontinuity a neighbourhood $U$ of $[C] \in \binimal_{x_1, x_2}$ 
parametrising $2$-free smooth rational curves. For a $2$-free rational
curve, the evaluation morphism $\ev$ is smooth 
in the complement of the distinguished divisors $\sigma_i$ \cite[II,
Prop.3.5.1]{Kol96}. Thus if we denote by $R \subset  \Univ_{x_1, x_2}$
the ramification divisor 
of $\ev$, $\sigma_1$ is the unique irreducible component of $R$ 
containing the point $C \cap \sigma_1$. Thus
$\Gamma$ is not contained in the ramification divisor 
of $\ev$.

Since $q(\Gamma)$ is a curve, there exists by Lemma \ref{l:basechange} 
a finite base change $T \rightarrow q(\Gamma)$ 
with $T$ a smooth curve, 
such that the normalisation $S$ of the fibre product 
$T \times_{\binimal_{x_1, x_2}}  \Univ_{x_1, x_2}$ is a surface
with a conic bundle structure $\holom{\varphi}{S}{T}$
satisfying the conditions of Lemma \ref{l:computation}.
After a further base change we may suppose that there exists a
$\varphi$-section  $\Gamma_1$ 
that maps onto $\Gamma$.
Note that since we obtained $S$ by a base change from $\Univ_{x_1, x_2}$,
the ramification divisor of the map $\mu: S \rightarrow
\Univ_{x_1, x_2}$ is contained in the $\varphi$-fibres, 
i.e.\ its image by $\varphi$ has dimension $0$.
In particular $\Gamma_1$ is not contained in this ramification locus. 

Since the rational curve $C$ is smooth and $2$-free, the universal
family $\Univ_{x_1, x_2}$ is smooth in a  
neighbourhood of $C \cap \sigma_1$. Thus $\sigma_1$ is a Cartier
divisor in a neighbourhood of $C \cap \sigma_1$,
and we can use the projection formula to see that  
$$
\Gamma_1 \cdot \mu^* \sigma_1 =
\mu_*(\Gamma_1) \cdot \sigma_1 > 0. 
$$ 
In particular $\Gamma_1$ is not disjoint from the distinguished sections in the conic bundle $S \rightarrow T$. 
Let now \holom{\epsilon}{\hat S}{S}
be the minimal resolution of singularities,
and $\hat \Gamma_1$ the proper transform of $\Gamma_1$. 
Since the distinguished sections are in the smooth locus of $S$,
the section $\hat \Gamma_1$ is not disjoint from the distinguished
sections of $\hat S \rightarrow T$. We shall now show that
$$
(\hat \Gamma_1)^2 \leq 0,
$$
which is a contradiction to Lemma~\ref{l:computation}.
%\ref{c:disjoint-zero}.

Denote by $f: \hat \Gamma_1 \rightarrow l$ the restriction of 
$\ev \circ \mu \circ \epsilon: \hat S \to X$.
Since $\hat \Gamma_1$ is not in the ramification locus of 
$\mu \circ \epsilon$ and 
$\Gamma$ is not in the ramification divisor of $\ev$, the tangent map
$$
T_{\hat S}|_{\hat \Gamma_1} \rightarrow f^* T_X|_{l}
$$
is generically injective. Since $\hat \Gamma_1$ is a $\varphi \circ \epsilon$-section, we have an isomorphism 
\begin{equation}
\label{isomTgN}
T_{\hat S/T}|_{\hat \Gamma_1} \simeq N_{\hat \Gamma_1/\hat S}.
\end{equation}
Since $l$ has the standard splitting type \eqref{splitstandard} 
we have a (unique) trivial quotient $f^* T_X|_{l} \twoheadrightarrow
\sO_{\hat \Gamma_1}$, and thanks to \eqref{isomTgN}
we are done if we prove that the natural map
$$
T_{\hat S/T}|_{\hat \Gamma_1} \hookrightarrow T_{\hat S}|_{\hat \Gamma_1} \rightarrow f^* T_X|_{l} \twoheadrightarrow \sO_{\hat \Gamma_1} 
$$
is not zero. It is sufficient to check this property for a general
point in $\hat \Gamma_1$, and since $\hat \Gamma_1 \rightarrow \Gamma$
is generically \'etale, it is sufficient to check that for a general $x
\in \Gamma$, the natural map 
$$
T_{\Univ_{x_1, x_2}/\binimal_{x_1, x_2}, x} \rightarrow \ev^* (T_{X,\ev(x)}) 
$$
does not have its image into the ample part  $\ev^* (T_X|_{l, \ev(x)}^+)$. Yet by \eqref{idsubspaces}
we know that $(q^* T_{\binimal_{x_1, x_2}})_{x}$ maps into the ample part. Thus if $T_{\Univ_{x_1, x_2}/\binimal_{x_1, x_2}, x}$
also maps into the ample part, then the tangent map
$$
T_{\Univ_{x_1, x_2},x} \rightarrow \ev^* (T_{X, \ev(x)})
$$
cannot be surjective. Since $\Gamma$ is not contained in the
ramification locus of $\ev$ this is a contradiction.
\end{proof}

\section{Proof of the main theorem}

\begin{paragraph}{Proof of Theorem~\ref{t:miyaoka}.}
If $X \simeq \PP^n$ we are done, so suppose that this is not the
case. Then consider the family of minimal rational curves $\minimal$
constructed in Section~\ref{S:work} and the associated total VMRT
$\vmrt$. 
Denote by $d \in \N$ the degree of a general VMRT $\vmrtx \subset
\PP(\Omega_{X, x})$.  

\medskip\noindent
{\em Step 1. Using the family $\binimal^\circ$.} 
In this step we prove that  
\begin{equation} \label{equationcool}
\vmrt \sim_\Q d (\zeta - \frac{1}{n} \pi^* K_X),
\end{equation}
where $\zeta$ is the tautological divisor class on $\PP(\Omega_X)$. 
Note that $\PP(\Omega_X)$ has Picard number two, so we can always write
$$
\vmrt \sim_\Q a \zeta + b \frac{-1}{n} \pi^* K_X
$$
with $a, b \in \Q$. Let now $\binimal^\circ$ be the family of rational
curves constructed in Section~\ref{S:work}, 
and let $\tilde C$ be the lifting of a curve $C \in \binimal^\circ$. By
Proposition~\ref{p:zero} we have 
$\vmrt \cdot \tilde C=0$. Since by the definition of $\tilde C$ one
has $\zeta \cdot \tilde C = -2$ 
and $-\frac{1}{n} \pi^* K_X \cdot \tilde C=2$, it follows that
$a=b$. Since $\vmrtx = \vmrt|_{\PP(\Omega_{X, x})} \sim_\Q d
\zeta|_{\PP(\Omega_{X, x})}$, we have $a=b=d$. This proves
\eqref{equationcool}. 

\medskip\noindent
{\em Step 2. Bounding the degree $d$.}
Denote by $\minimalz \subset \minimal$ the 
open set parametrising smooth standard rational curves
in $\minimal$.
We define an injective map 
$$
j: \minimalz \hookrightarrow \Rat(\PP(\Omega_X))
$$
by mapping a curve $l$ to the image $\tilde l$ of the morphism
$s: l \rightarrow \PP(\Omega_X)$
defined by the invertible quotient $\Omega_X|_l \rightarrow
\Omega_l$. We denote by $\minimalzt$ the image of $j$.
Let us start by showing that $\minimalzt$ 
is dense in an irreducible component of $\Rat(\PP(\Omega_X))$. Arguing
by contradiction we suppose that $\minimalzt$ 
is contained in an irreducible component $\mathcal{R}$ 
of dimension strictly larger than $2n-3$. The projection $\pi$ defines
a map $\pi_*$ between the spaces 
of rational curves,
and by construction $\pi_*(\mathcal{R})$ contains $\minimalz$. Since $\minimalz$ is dense in an irreducible component
of $\Rat(X)$ we obtain that (up to replacing $\mathcal{R}$ by a Zariski open set) we have a map 
$\pi_*: \mathcal{R} \rightarrow \minimalz$. Since $\dim \mathcal{R}>2n-3$ this map has positive-dimensional fibres, 
so a general curve $\tilde l$ deforms in $\PP(\Omega_X|_l)$. Yet this is impossible since
$$
N_{\tilde l/ \PP(\Omega_X|_l)} \simeq \sO_{\PP^1}(-2) \oplus \sO_{\PP^1}(-1)^{\oplus n-2}.
$$
By construction the lifted curves $\tilde l$ 
are contained in $\vmrt$. Thus the open set $\tilde {\minimal}_0
\subset \Rat(\PP(\Omega_X))$  
is actually an open set in $\Rat(\vmrt)$. Since $\vmrt \subset
\PP(\Omega_X)$ is a hypersurface, 
the algebraic set $\vmrt$ has lci singularities. Thus we can apply \cite[II, Thm.1.3, Thm.2.15]{Kol96} and obtain
$$
2n-3 = \dim \tilde {\minimal}_0 \geq - K_{\vmrt} \cdot \tilde l + (2n-2) - 3.
$$
We thus have $- K_{\vmrt} \cdot \tilde l \leq 2$.

Now by construction we have $-\frac{1}{n} \pi^* K_X \cdot \tilde l=1$ and
$\zeta \cdot \tilde l=-2$.
Since $K_{\PP(\Omega_X)} = 2\pi^* K_X - n \zeta$, the adjunction
formula and \eqref{equationcool} yield 
$$
2 \geq - K_{\vmrt} \cdot \tilde l = -(K_{\PP(\Omega_X)}+\vmrt) \cdot \tilde l = d.
$$

\medskip\noindent
{\em Step 3. Conclusion.}
If $d=1$ or $d=2$ but $\vmrtx$ is reducible, we obtain a contradiction
to \cite[Thm.1.5]{Hwa07} 
(cf.\ also \cite[Thm.3.1]{Ara06}).
If $d=2$ and $\vmrtx$ is irreducible, $\vmrtx$ is normal
\cite[II,Ex.6.5(a)]{Har77}, and therefore isomorphic to
its normalisation $\minimal_x$ which is smooth (see
\S\ref{s:minimal}).
It is thus a smooth quadric and we conclude by \cite[Main Thm.]{Mok08}.
\qed
\end{paragraph}

\begin{remark} \label{remarkgap}
Let us explain the difference of our proof with Miyaoka's approach:
in the notation of Section~\ref{S:work}, he considers the family
$\binimal_{x_1, x_2}$. 
As we have seen above the evaluation map $\holom{\ev}{\Univ_{x_1,
    x_2}}{X}$ is generically finite and his goal is to prove 
that $\ev$ is  birational. He therefore analyses the preimage
$\fibre{\ev}{l_1 \cup l_2}$, 
where the $l_i \subset X$ are general minimal curves passing through
$x_i$ respectively such that $[l_1 \cup l_2] \in \binimal_{x_1, x_2}$. 
If $\Gamma \subset \fibre{\ev}{l_1 \cup l_2}$ is an irreducible curve
mapping onto $l_1$ one can make a case distinction: 
if $q(\Gamma)$ is a curve that is not contained in the discriminant
locus $\Delta \subset \binimal_{x_1, x_2}$ 
(Case $\bf C$  in \cite[p.227]{Miy04}) Miyaoka makes a very interesting 
observation which we stated as Proposition~\ref{p:separation}. However
the analysis of the `trivial' case 
(Case $\bf A$ in \cite[p.227]{Miy04}) where $q(\Gamma)$ is a point is
not correct: it is not clear that 
$q(\Gamma) = [l_1 \cup l_2]$, because
there might be another curve in $\binimal_{x_1, x_2}$ which is of the
form $l_1 \cup l_2'$ with $l_2 \neq l_2'$. This possibility is an
obvious obstruction to the birationality of $\ev$ and invalidates
\cite[Cor.3.11(2), Cor.3.13(1)]{Miy04}. The following example shows
that this possibility does indeed occur in certain cases.
\end{remark}

\begin{example}
Let  $H \subset \PP^n$ be a hyperplane and $A \subset H \subset \PP^n$
a projective manifold $A$ of dimension $n-2$  
and degree $3 \leq a \leq n$. Let $\holom{\mu}{X}{\PP^n}$ be the
blow-up of $\PP^n$ along $A$.  
Then $X$ is a Fano manifold \cite[Rem.4.2]{Miy04} and $-K_X \cdot C
\geq n$ for every rational curve $C \subset X$ 
passing through a {\em general} point
(the $\mu$-fibres are however rational curves with $-K_X \cdot C=1$).
The general member of a family of minimal rational curves $\minimal$ 
is the proper transform of a line that intersects $A$. 
Consider the family $\binimal$ whose general member is the strict
transform  of a reduced, connected degree two curve $C$ such that $A
\cap C$ is a finite scheme of length two. 
For general points $x_1, x_2 \in X$ the (normalised) universal family
$\Univ_{x_1, x_2} \rightarrow \binimal_{x_1, x_2}$ is a conic bundle
and the evaluation map $\ev: \Univ_{x_1, x_2} \rightarrow X$ is
generically finite. We claim that $\ev$ is not birational. 

\noindent
\emph{Proof of the claim.}
For simplicity of notation we denote by $x_1, x_2$ also the
corresponding points in $\PP^n$. 
Let $l_1 \subset \PP^n$ be a general line through $x_1$ that
intersects $A$. Since $x_2 \in \PP^n$ is general there exists 
a unique plane $\Pi$ containing $l_1$ and $x_2$. Moreover the
intersection $\Pi \cap A$ consists of exactly $a$ points, one of 
them the point $A \cap l_1$. For every point $x \in \Pi \cap A$
other than $A \cap l_1$, there exists a unique line $l_{2,x}$ 
through $x$ and $x_2$. By Bezout's theorem $l_1 \cup l_2$ is
connected, so its proper transform belongs to $\binimal_{x_1, x_2}$. 
Yet this shows that $\fibre{\ev}{l_1}$ contains $a-1>1$ copies of
$l_1$, one for each point $x \in \Pi \cap A \setminus l_1 \cap A$.
This proves the claim.
\qed

Let us conclude this example by mentioning that the conic bundle
$\Univ_{x_1, x_2} \rightarrow \binimal_{x_1, x_2}$ does not satisfy
the symmetry conditions of Lemma~\ref{l:computation}.
\end{example}

%\bibliographystyle{alpha}
%\bibliography{biblio}

\def\cprime{$'$}

\end{document}